\documentclass[12pt]{amsart}
\usepackage{amsthm,amssymb,amsmath,amscd,epic,eepic}
\usepackage[all]{xy}
\usepackage{graphicx}
\usepackage{enumerate}
\usepackage{epstopdf}

\newif\ifdebug                                                      %
\debugfalse

\numberwithin{equation}{section}

\newtheorem{theorem}{Theorem}[section]

\newtheorem{proposition}[theorem]{Proposition}

\newtheorem{corollary}[theorem]{Corollary}

\theoremstyle{remark}

\newtheorem{example}[theorem]{Example}
\newtheorem{definition}[theorem]{Definition}

\def    \inv    {^{-1}}

\def    \R	{{\mathbb R}}
\def    \P	{{\mathbb P}}
\def    \N	{{\mathbb N}}
\def    \Q	{{\mathbb Q}}
\def    \C	{{\mathbb C}}
\def    \Z  {{\mathbb Z}}
\def    \e  {{\epsilon }}
\def    \l  {{\lambda }}
\def \li {\textrm{Lie}}

\def \col {{\mathcal O}_{\lambda}}

\def\eor{\unskip\ \hglue0mm\hfill$\diamond$\smallskip\goodbreak}

\begin{document}

\title[The Gromov width of orbits of the symplectic group]{The Gromov width of coadjoint orbits of the symplectic group}

\author[Iva Halacheva]{Iva Halacheva}
\address{iva.halacheva@utoronto.ca \\ University of Toronto, 40 St. George Street, Toronto, M5S 2E4, Ontario, Canada}
\author[Milena Pabiniak]{Milena Pabiniak}
\address{pabiniak@math.uni-koeln.de\\Mathematisches Institut, Universit\"at zu K\"oln, Weyertal 86-90, D-50931 K\"oln, Germany}

\begin{abstract} In this work we prove that the Gromov width of a coadjoint orbit of the symplectic group through a regular point $\lambda$, lying on some rational line, is at least equal to:
$$\min\{\, \left|\left\langle \alpha^{\vee},\lambda \right\rangle \right|; \alpha^{\vee} \textrm{ a  coroot}\}.$$
Together with the results of Zoghi and Caviedes concerning the upper bounds, this establishes the actual Gromov width. 
This fits in the general conjecture that for any compact connected simple Lie group $G$, the Gromov width of its coadjoint orbit through $\lambda \in \li (G)^*$ is given by the above formula.
The proof relies on tools coming from symplectic geometry, algebraic geometry and representation theory:
we use a toric degeneration of a coadjoint orbit to a toric variety whose polytope is the string polytope arising from a string parametrization of elements of a crystal basis for a certain representation of the symplectic group.
\end{abstract}

\maketitle

\section{Introduction}
The non-squeezing theorem of Gromov motivated the question of finding the biggest ball that could be symplectically embedded into a given symplectic manifold $(M, \omega)$.
Consider the ball of {\bf capacity} $a$:
$$ B^{2N}_a = \big \{ (x_1,y_1,\ldots,x_N,y_N) \in  \R^{2N} \ \Big | \ \pi \sum_{i=1}^N (x_i^2+y_i^2) < a \big 
\} \subset \R^{2N}, $$
with the standard symplectic form
$\omega_{std} = \sum dx_j \wedge dy_j$.
The \textbf{Gromov width} of a $2N$-dimensional symplectic manifold $(M,\omega)$
is the supremum of the set of $a$'s such that $B^{2N}_a$ can be symplectically
embedded in $(M,\omega)$. It follows from Darboux's theorem that the Gromov width 
is positive unless $M$ is a point.

Coadjoint orbits form an important class of symplectic manifolds.
Let $K$ be a compact Lie group.
It acts on itself by conjugation
$$K \ni g:K \rightarrow K,\;\;\;g(h)=g h g \inv.$$
Associating to $g \in K$ the derivative of the above map, taken at the identity, $dg_e \colon T_eK \rightarrow T_eK$,
one obtains the adjoint action of $K$ on $\mathfrak{k}=\li (K)=T_eK$.
This induces the action of $K$ on $\mathfrak{k}^*=\li (K)^*$, the dual of its Lie algebra, called the {\bf coadjoint action}. Each orbit 
$\mathcal{O}\subset \li (K)^*$ of the coadjoint action
is naturally equipped with the {\bf Kostant-Kirillov-Souriau symplectic form}:
$$\omega_{\xi}(X^\#,Y^\#)=\langle \xi, [X,Y]\rangle,\;\;\;\xi \in \mathcal{O} \subset \li (K)^*,\;X,Y \in \li (K),$$
where $X^\#,Y^\#$ are the vector fields on $\li (K)^*$ corresponding to $X,Y \in \li (K),$ induced by the coadjoint $K$ action. 
The coadjoint action of $K$ on $\mathcal{O}$ is Hamiltonian, 
and the momentum map is the inclusion $\mathcal{O}\hookrightarrow \li (K)^*$. 
Every coadjoint orbit intersects a chosen positive Weyl chamber in a single point. Therefore there is a bijection between the coadjoint orbits and points in the positive Weyl chamber. Points in the interior of the positive Weyl chamber are called {\bf regular} points. The orbits corresponding to regular points are of maximal dimension. They are diffeomorphic to $K/T$, for $T$ a maximal torus of $K$, and are called {\bf generic orbits}.
For example, when $K=U(n,\C)$, the group of (complex) unitary matrices, a coadjoint orbit can be identified with the set of Hermitian matrices with
a fixed set of eigenvalues. The generic orbits are diffeomorphic to the manifold of full flags in $\C^n$.

In this note we concentrate on the (compact) symplectic group $K=\text{Sp}(n)=U(n,\mathbb{H})$. 
The main result of this manuscript is the following theorem.
\begin{theorem}\label{gromovmain}
Let $M:=\mathcal{O}_{\lambda}$ be the coadjoint orbit of $K=Sp(n)$
through a regular point $\lambda$ lying on some rational line in $\mathfrak k^*$,
 equipped with the Kostant-Kirillov-Souriau symplectic form.
The Gromov width of $M$ is at least the minimum 
$$\min\{\, \left|\left\langle \alpha^{\vee},\lambda \right\rangle \right|: \alpha^{\vee} \textrm{ a  coroot}\}.$$
\end{theorem}
If $\lambda= \lambda_1 \omega_1+\ldots + \lambda_n \omega_n$ where $\omega_1,\ldots,\omega_n$ are the fundamental weights, and $\lambda_j >0$, then the above minimum is equal to, as we explain in Section \ref{main proof}, $\min\{\lambda_1,\ldots,\lambda_n\}$.

This particular lower bound is important because it coincides with the known upper bound.
Zoghi proved in \cite{Zoghi} that for a compact connected simple Lie group $K$, the above formula gives an upper bound for the Gromov width of a regular indecomposable coadjoint $K$-orbit through $\lambda$ (\cite[Proposition 3.16]{Zoghi}). 
This result was later extended to non-regular orbits by Caviedes.
\begin{theorem}\cite[Theorem 8.3]{Caviedes}\cite[Proposition 3.16, regular orbits]{Zoghi}\label{caviedes}
Let $K$ be a compact connected simple Lie group.
The Gromov width of a coadjoint orbit $\col$ through $\lambda$, equipped with the Kostant-Kirillov-Souriau symplectic form, is at most
$$\min\{\, \left|\left\langle \alpha^{\vee},\lambda \right\rangle \right|:\  \alpha^{\vee} \textrm{ a  coroot and }\left\langle \alpha^{\vee},\lambda \right\rangle \neq0\}.$$
\end{theorem}

Putting these results together we obtain the following corollary.
\begin{corollary}\label{cor gromovwidth}
The Gromov width of a coadjoint orbit $\col$ of $Sp(n)$ through a regular point $\lambda$ lying on some rational line in $\mathfrak k^*$,
is exactly 
$$\min\{\, \left|\left\langle \alpha^{\vee},\lambda \right\rangle \right|:\  \alpha^{\vee} \textrm{ a  coroot}\}.$$
\end{corollary}

What adds importance to our result is the fact that it is a special case of a general conjecture about the Gromov width of coadjoint orbits of compact Lie groups. Namely, it has been conjectured, and by now proved in many cases, that for any compact connected simple Lie group $K$, the Gromov width of its coadjoint orbit through $\lambda \in \li (K)^*$ is given by 
 the formula from Theorem \ref{caviedes},
i.e. it is the minimum over the positive results of pairings of $\lambda$ with coroots in the system.
Karshon and Tolman in \cite{KarshonTolman}, and independently Lu in \cite{Lu2}, showed that the Gromov width of complex Grassmannians (which are degenerate coadjoint orbits of $U(n,\C)$) is given by the above formula. Combining the results of Zoghi (\cite{Zoghi}) and Caviedes (\cite{Caviedes}) about upper bounds, and the results of the second author (\cite{Pabiniak}) about lower bounds, one proves that the Gromov width of (not necessarily regular) coadjoint orbits of
 $U(n,\C)$, $SO(2n,\R)$ and $SO(2n+1,\R)$ is also given by that formula. (The result for $SO(2n+1,\R)$ works only for orbits satisfying one mild technical condition; see \cite{Pabiniak} for more details).

To prove the main result we use tools from symplectic geometry, algebraic geometry and representation theory.
Here is a brief outline. 
Using the work of \cite{HK} one can construct a toric degeneration from the given coadjoint orbit $\mathcal{O}_{\lambda}$ to a toric variety.  By ``pulling back" the toric action from the toric variety one equips (an open dense subset of) $\mathcal{O}_{\lambda}$ with a toric action and can use its flow to construct embeddings of balls. 
If $\lambda$ is a dominant weight, there exists a particularly nice toric degeneration to a toric variety whose associated Newton-Okounkov body is the string polytope parametrizing a crystal basis for (the dual of) the irreducible representation with highest weight $\lambda$ (\cite{K}). Such string polytopes have been studied by Littelmann in \cite{L}, and using his work we prove Theorem \ref{gromovmain} for orbits $\mathcal{O}_{\lambda}$ with $\lambda$ a dominant weight. We then further extend this result to any regular $\lambda$ lying on a rational line in $\mathfrak k^*$.

The techniques used in this paper could be applied to other compact connected simple Lie groups to obtain a lower bound for the Gromov width by studying the structure of (more general) string polytopes. We do not pursue this idea here for the following reason. As the formula for the conjectured Gromov width is given in purely Lie-theoretic language, we believe that there should be a way of proving the (lower bound part of the) conjecture for all groups at once, by a proof described in purely Lie-theoretic language.

{\bf Organization.} In Section \ref{tools} we introduce the tools that are used in Section \ref{main proof} to prove the main result.

{\bf Acknowledgements.} The authors are very grateful to Kiumars Kaveh for explaining his work to us.
We also thank Yael Karshon, Joel Kamnitzer and Alexander Caviedes for useful discussions. We are very grateful to the anonymous referee for their corrections (a missing assumption that $\lambda$ is on a rational line) and pointing out misprints in the first version, as well as for their comments which improved the exposition of this paper.

The first author was supported by an NSERC Alexander Graham Bell CGS D and a Queen Elizabeth II graduate scholarship.
The second author was supported by the Funda\c{c}\~ao para a Ci\^encia e a Tecnologia (FCT), Portugal:
fellowship SFRH/BPD/87791/2012 and projects PTDC/MAT/117762/2010, EXCL/MAT-GEO/0222/2012.


\section{Tools} \label{tools}

\subsection{ Using a toric action to construct symplectic embeddings of balls.}
Toric geometry proves to be very helpful in finding lower bounds for the Gromov width. 
When a manifold $(M,\omega)$ is equipped with a Hamiltonian (so also effective) action of a torus 
$T$, one can use the flow of the vector field generated by this action to construct explicit embeddings of balls and 
therefore to obtain a lower bound for the Gromov width (a construction by Karshon and Tolman in \cite{KarshonTolman}).
If additionally the action is {\bf toric}, that is $\dim T= \frac 1 2 \dim M$, 
then more constructions are available
(see for example: \cite{Traynor}, \cite{Schlenk}, \cite{LMS}). 

Recall that a Hamiltonian action of a torus $T$ on a symplectic manifold $(M,\omega)$ 
gives rise to a momentum map $\mu \colon M \rightarrow \li (T)^*=:\Lambda_{\R}$, from $M$ to the dual of the Lie algebra of $T$, 
which we denote by $\Lambda_{\R}$. 
This map is unique up to a translation in $\Lambda_{\R}$. 
A manifold $M$ equipped with a Hamiltonian $T$ action is often called a {\bf Hamiltonian $T$-space}.
When $M$ is compact, the image $\mu(M)$ is a Delzant polytope. 
Identifying $ \Lambda_{\R}$ with $\R^{\dim T}$, we can view $\mu(M)$ as a polytope in $\R^{\dim T}$. 
Such an identification is not unique: it depends on the choice of a splitting of $T$ into a product of circles, 
and on the choice of an identification of the Lie algebra of $S^1$ with the real line $\R$. 
Changing the splitting of $T$ results in applying a $GL(\dim T, \Z)$ transformation to $\R^{\dim T}$, 
while changing the identification $\li(S^1) \cong \R$ results in rescaling. 
In this work, 
 $S^1 = \R /  \Z$, that is, the exponential map $\exp \colon \R=\li(S^1) \rightarrow S^1$ is given by $t \mapsto e^{2 \pi i t}$. With this convention, the momentum map for the standard 
$S^1$-action on $\C$ by rotation with speed $1$ is given (up to the addition of a constant) by 
$z \mapsto -\pi |z|^2$.

Consider the standard $T^n=(S^1)^n$ action on $\C^n$ where each circle rotates a corresponding copy of $\C$ with speed $1$, with
a momentum map 
$$(z_1,\ldots,z_n) \mapsto -\pi (|z_1|^2,\ldots, |z_n|^2).$$
The image of the $n$-dimensional ball of capacity $a$ (radius $\sqrt{\frac{a}{ \pi}}$)
centered at the origin is $(-1)$ times the standard simplex of size $a$
$$\Delta^n(a):=\left\{(x_1,\ldots,x_n) \in \R^n_{\geq 0}\,|\,\sum_{k=1}^{n}x_k < \pi \cdot (\sqrt{a/\pi})^2=a\right\}.$$
Moreover, simplices embedded in the momentum map image signify the existence of embeddings of balls, as the following result explains.

\begin{proposition}\cite[Proposition 1.3]{Lu}\cite[Proposition 2.5]{Pabiniak}\label{embedding}
For any connected, proper (not necessarily compact) Hamiltonian $T^n$-space $M^{2n}$ of dimension $2n$ let 
\begin{align*}
\mathcal{W}(\Phi(M))= \sup \{&a>0\,|\, \exists \; \Psi \in GL(n,\Z), x \in
\R^n,\textrm{ such that }\\
&\Psi (\Delta^n(a))+\,x \subset \Phi(M) \},
\end{align*}
where $\Phi$ is some choice of momentum map. Then the Gromov width of $M$ is at least $\mathcal{W}(\Phi(M))$.
 \end{proposition}
 
 \subsection{Coadjoint orbits as flag varieties}
 \label{sec:coadj}
 Coadjoint orbits of compact Lie groups can be viewed as flag manifolds of complex reductive groups.
 This interpretation allows us to later construct toric degenerations of coadjoint orbits (Section \ref{toric degeneration}). 
 
 Let $G$ be a connected reductive group over $\C$ and $B$ a Borel subgroup.
 Denote by $\Lambda$ the weight lattice of $G$ and by $\Lambda^+$ the dominant weights. 
Let $K$ be the compact form of $G$ and $T$ its maximal torus. 
A generic coadjoint orbit of $K$, $K/T$, is diffeomorphic to the flag manifold $G/B$.
 To equip the manifold $G/B$ with a symplectic structure, fix $\lambda \in \Lambda^+$ and let $V_{\lambda}$ denote the finite dimensional irreducible representation of $G$ with highest weight $\lambda$. There exists a very ample $G$-equivariant line bundle $\mathcal{L}_{\lambda}$ on $G/B$ whose space of sections $H^0(G/B, \mathcal{L}_{\lambda})$ is isomorphic to $V_{\lambda}^*$ (Borel-Weil Theorem). 
 Embed $G/B$ into $ \P(H^0(G/B, \mathcal{L}_{\lambda})^*)$ (the Kodaira embedding),
 and use this embedding to pull back to $G/B$ the Fubini-Study symplectic structure.
If $\omega_{\lambda}$ denotes the symplectic structure on $G/B$ obtained this way, then
$(G/B, \omega_{\lambda} )$
 is symplectomorphic to the coadjoint orbit $\col$ with the Kostant-Kirillov-Souriau symplectic structure defined in the Introduction.
 
 In this manuscript, $G=\text{Sp}(2n,\C)$ and $K=\text{Sp}(n)=U(n,\mathbb{H})$.

\subsection{Obtaining a toric action via a  toric degeneration.}\label{toric degeneration}
Coadjoint orbits of a compact Lie group $K$ are naturally equipped with a Hamiltonian action of a maximal torus of $K$.
This action, however, is rarely toric. We remark that for $U(n,\C), SO(n,\R)$ a toric action can be constructed by Thimm's trick (\cite{Pabiniak}).

To obtain a toric action on a dense open subset of a coadjoint orbit of $\text{Sp}(n)$, we apply a method developed by Harada-Kaveh in \cite{HK} using toric degenerations.
We briefly sketch the main ingredients of their construction and for details direct the reader to \cite{HK}.

Consider the situation where $X$ is a $d$-dimensional projective algebraic variety, $\mathcal{L}$ an ample line bundle over $X$, and let $\C(X)$ denote the field of rational functions on $X$. 
Given a valuation $\nu \colon \C(X) \setminus \{0\} \rightarrow \Z^d$ with one-dimensional leaves, one builds an additive semi-group 
$$S=S(X,L,v,h)=\bigcup_{k>0}\{(k,v(f/h^k))\,|\, f \in L^{\otimes k}\setminus \{0\} \}.$$
and 
a convex body
$$\Delta(S)=\overline{\textrm{conv}(\bigcup_{k>0}\{x/k\,|\,(k,x) \in S\})},$$
in $\R^d$, called an {\bf Okounkov (or Newton-Okounkov) body}. 
Here $h$ is a fixed section of $\mathcal{L}$ and  $L^{\otimes k}$ denotes the image of the $k$-fold product $L\otimes \ldots \otimes L$ in $H^0(X,\mathcal{L}^{\otimes k})$.

  \begin{theorem}\cite[Proposition 5.1 and Corollary 5.3]{And}, \cite[Corollary 3.14]{HK}
\label{existence degeneration}
With the notation as above, assume in addition that $S$ is finitely generated. 
Then there exists a finitely generated, $\N$-graded, flat $\C[t]$-subalgebra $\mathcal{R} \subset \C(X)[t]$ inducing a flat family 
$\pi \colon \mathfrak{X} =Proj\, \mathcal{R} \rightarrow \C$ such that:
\begin{itemize}
\item For any $z \neq 0$ the fiber $X_z=\pi^{-1}(z)$ is isomorphic to $X=Proj \,\C(X)$, i.e. $\pi^{-1}(\C \setminus \{0\})$ is isomorphic to $X \times (\C \setminus \{0\})$.
\item The special fiber $X_0=\pi^{-1}(0)$ is isomorphic to $Proj\, \C[S]$ and is equipped with an action of $(\C^*)^d$, where $d=\dim_\C X$. 
The normalization of the variety $Proj \,\C[S]$ is the toric variety associated to the rational polytope $\Delta(S)$.
\end{itemize}
\end{theorem}

Fix a Hermitian structure on the very ample line bundle $\mathcal{L}$ and equip $X$ with the symplectic structure $\omega$ induced from the Fubini-Study form on $\P(H^0(X,\mathcal{L})^*)$ via the Kodaira embedding.

\begin{theorem}\cite[Theorem 3.25]{HK}\label{existence integrable system}
With the notation as above, assume in addition that $(X, \omega)$ is smooth and that the semigroup $S$ is finitely generated.
 Then:
\begin{enumerate}
\item There exists an integrable system $\mu=(F_1,\ldots, F_d) \colon X \rightarrow \R^d$ on $(X, \omega)$ in the sense of \cite[Definition 1]{HK}, 
and the image of $\mu$ coincides with the Newton-Okounkov body $\Delta=\Delta(S)$.
\item The integrable  system generates a torus action on the inverse image under $\mu$ of the interior of the moment polytope $\Delta$.\footnote{In fact the action is defined on the set $U$ introduced in \cite[Definition 1]{HK}, which contains but might be strictly bigger than the inverse image under $\mu$ of the interior of the moment polytope $\Delta$.}
\end{enumerate}
\end{theorem}

In this manuscript we use valuations (with one dimensional leaves) coming from the following examples.
\begin{example}\label{highest term val} (\cite[Example 3.3]{HK})
Fix a linear ordering on $\Z^d$. 
Let $p$ be a smooth point in $X$, and let $u_1,\ldots,u_{d}$ be a regular system of parameters in a neighborhood of $p$.
Using this system, we can construct the lowest and the highest term valuations on $\C(X)$: the {\bf lowest (resp. highest) term valuation} $v_{low}$ (resp. $v_{high}$) assigns to each $f (u_1,\ldots,u_d)=\sum_{j=(j_1,\ldots,j_d)}c_j u_1^{j_1}\ldots u_d^{j_d}\in \C(X)$ a $d$-tuple of integers 
which is the smallest (resp. biggest) among $j=(j_1,\ldots,j_d)$ with $c_j\neq 0$, in the fixed order. 
To a rational function $f/h \in \C(X)$ this valuation assigns $v_{low}(f)-v_{low}(h)$ (resp. $v_{high}(f)-v_{high}(h)$).
Both of these valuations have one dimensional leaves. 
\eor
\end{example}

\begin{example}\label{highest term for flags}
What will be very relevant for this manuscript is a special case of the previous Example.
In the situation we consider here, $X$ is the flag variety $G/B$ of the symplectic group $G=\text{Sp}(2n, \C)$, with $B$ a fixed Borel subgroup of $G$. 
Choose a reduced decomposition $\underline{w_0}=(\alpha_{i_1},\ldots, \alpha_{i_N})$ of the longest word in the Weyl group $w_0=s_{\alpha_{i_1}}\cdots s_{\alpha_{i_N}}$, where $s_{\alpha_i}$ is the reflection through the hyperplane orthogonal to the simple root $\alpha_i$:
	$$s_{\alpha_i}(\beta)=\beta-2\frac{\left\langle\beta,\alpha_i\right\rangle}{\left\langle\alpha_i,\alpha_i\right\rangle}\alpha_i$$ 
It defines a sequence of (Schubert) subvarieties (a Parshin point)
$$\{o\}=X_{w_N}\subset \ldots \subset X_{w_0}=X,$$
where $X_{w_k}$ is the Schubert variety corresponding to the Weyl group element $w_k=s_{\alpha_{i_{k+1}}}\cdots s_{\alpha_{i_N}}$, and $\{o\}$ is the unique $B$-fixed point in $X$.  This sequence of varieties, in turn, gives rise to a regular system of parameters $u_1,\ldots,u_{d}$, in which $X_{w_k}=\{u_1=\ldots=u_k=0\}$ (see Section 2.2 of \cite{K}). Following Kaveh (\cite{K}), we denote the associated highest term valuation (as in Example \ref{highest term val}) on $\C(X) \setminus \{0\}$ by $v_{\underline{w_0}}$.\label{our valuation}\eor
\end{example}

\subsection{Crystal bases and Newton-Okounkov bodies.}

We now return to analyzing the flag manifold. With $G$, $B$, $\lambda \in \Lambda^+$, $V_{\lambda}$, and $\mathcal{L}_{\lambda}$ as in Section \ref{sec:coadj}, recall that $G$ acts on the space of sections $H^0(G/B,\mathcal{L}_{\lambda})$ giving a representation isomorphic to the dual representation $V^{\ast}_{\lambda}$.
There exists a particular toric degeneration of the flag variety $G/B$ for which the associated Okounkov body is the string polytope parametrizing the elements of a crystal basis of the representation $V^{\ast}_{\lambda}$. 
Before analyzing this toric degeneration, we recall some basic facts about crystal bases.

Let $I$ denote the Dynkin diagram, and $\{\alpha_i\}_{i \in I},\{\alpha^{\vee}_i\}_{i \in I}$ denote the simple roots and coroots respectively. 
We will look at the perfect basis for $V^{\ast}_{\lambda}$ coming from the specialization of Lusztig's canonical basis to $q=1$ for the quantum enveloping algebra, which Kaveh refers to as a crystal basis for $V^{\ast}_{\lambda}$ in \cite{K}. Note that this differs from Kashiwara's notion of crystal basis being the specialization at $q=0$. 

A {\bf perfect basis} for a finite-dimensional representation $V$ of $G$ is a weight basis $B_V$ of the vector space $V$ together with a pair of operators, called Kashiwara operators,
$\tilde{E}_{\alpha}, \; \tilde{F}_{\alpha}: B_V \rightarrow B_V\cup\{0\}$
 for each simple root $\alpha$, and maps $\tilde{\varepsilon}_{\alpha}, \tilde{\phi}_{\alpha}: V\setminus\{0\} \rightarrow \mathbb{Z}$ satisfying certain compatibility conditions. For further information, we refer the reader to \cite[Section 3.1]{K}.
 
One can associate to a perfect basis $B_V$ a directed labeled graph, called the {\bf crystal graph of the representation $V$}, whose vertices are the elements of $B_V \cup \{0\}$, and whose directed edges are labeled by the simple roots following the rule:  there is an edge from $b$ to $b'$ labeled $\alpha$ if and only if $\tilde{E}_{\alpha}(b)=b'$ (equivalently $\tilde{F}_{\alpha}(b')=b$). Also there is an edge from $b$ to $0$ if $\tilde{E}_{\alpha}(b)=0$, and from $0$ to $b$ if  $\tilde{F}_{\alpha}(b)=0$. The graphs obtained in this way are isomorphic for each perfect basis of the given $G$-representation $V$ (\cite[Theorem 5.55]{BK}).

A perfect basis $B_{\lambda}$ for the representation $V_{\lambda}$ with highest weight vector $v_{\lambda}$ can be obtained by considering the nonzero elements $gv_{\lambda}$ where $g$ is an element in the specialization to $q=1$ of the Lusztig canonical basis of the quantum enveloping algebra of $G$. The dual basis $B^{\ast}_{\lambda}$ is then a perfect basis for the dual representation $V^{\ast}_{\lambda}$, and will be referred to as the {\bf dual crystal basis} (see \cite[Lemma 5.50]{BK}).
The crystal $B_{\lambda}$ can be thought of as a combinatorial realization of $V_{\lambda}$ and reflects its internal structure. For more information about crystals see \cite{BK}, \cite{HoKa}, \cite{HeKa}.

There exists a nice parametrization of the elements of a (dual) crystal basis, called the {\bf string parametrization}, by integral points in $\Z^N$ where $N$ is the length of the longest word in the Weyl group $W$. This parametrization depends on a choice of a reduced decomposition $\underline{w_0}=(\alpha_{i_1},\ldots, \alpha_{i_N})$ of the longest word $w_0=s_{\alpha_{i_1}}\cdots s_{\alpha_{i_N}}$ in $W$:
$$\iota_{\underline{w_0}}\colon \coprod_{\lambda \in \Lambda^+} B_{\lambda}^* \rightarrow \Lambda^+ \times \Z_{\geq 0}^{N},$$
$$\iota_{\underline{w_0}}(B^{\ast}_{\lambda}) \subset \{\lambda\} \times  \Z_{\geq 0}^{N}.$$
The image of $\iota_{\underline{w_0}}$ is the intersection of a rational convex polyhedral cone $\mathcal{C}_{\underline{w_0}}$ in $\Lambda_{\R} \times \R^N$ with the lattice $ \Lambda \times \Z^N$. The projection of $\mathcal{C}_{\underline{w_0}}$ to $\R^N$ is a rational polyhedral cone in $\R^N$, called the {\bf string cone}, and will be denoted by $C_{\underline{w_0}}$. Littelmann analyzed in \cite{L} the image of string parametrizations (see also \cite[Theorem 1.1]{AB} and \cite[Theorem 3.4]{K}).

\begin{theorem}\cite[Proposition 1.5]{L}
For any dominant weight $\lambda$, the string parametrization is one-to-one. Moreover, $S_{\lambda}:=\iota_{\underline{w_0}}(B^{\ast}_{\lambda})$ is the set of integral points of a convex rational polytope $\Delta_{\underline{w_0}}(\lambda) \subset \R^N$ obtained as the intersection of the string cone, $C_{\underline{w_0}}$, and the $N$ half-spaces
$$x_k \leq \langle \lambda, \alpha_{i_k}^{\vee} \rangle - \sum_{l=k+1}^{N}\,x_l\,\langle \alpha_{i_l}, \alpha_{i_k}^{\vee} \rangle,\ \ k=1,\ldots, N.$$
\end{theorem}
(Note that in \cite{K} the symbol $\mathcal{C}_{\underline{w_0}}$ denotes a slightly different object: the projection of $\mathcal{C}_{\underline{w_0}}$ from \cite{K} to $\R^N$ is ``our"  $C_{\underline{w_0}}$ already intersected with the above $N$ half-spaces).
\begin{definition}\label{string polytope}
The polytope $\Delta_{\underline{w_0}}(\lambda) \subset \R^N$ is called the {\bf string polytope} associated to $\lambda$.
\end{definition}

For integral $\lambda$ the vertices of the polytope $\Delta_{\underline{w_0}}(\lambda)$ are rational, thus the cone over $\Delta_{\underline{w_0}}(\lambda)$,
$$\text{Cone}(\Delta_{\underline{w_0}}(\lambda))=\{(t,tx); t \in \R_{\geq 0}, x \in \Delta_{\underline{w_0}}(\lambda)\} \subset \R \times \R^N$$
is a strongly convex rational polyhedral cone. 

In \cite{K}, Kaveh observed the following relation between the string polytopes and Newton-Okounkov bodies associated to certain valuations that we have described in Section \ref{toric degeneration}.

\begin{theorem}\cite[Theorem 1]{K} \label{thm crystal okounkov}
The string parametrization for a dual crystal basis of $V_\lambda^*=H^0(G/B, \mathcal{L}_\lambda)$ is the restriction of the valuation $v_{\underline{w_0}}$ and the string polytope $\Delta_{\underline{w_0}}(\lambda)$ coincides with the Newton-Okounkov body of the algebra of sections of $\mathcal{L}_\lambda$ and the valuation $v_{\underline{w_0}}$.
\end{theorem}

\begin{corollary}\label{finitely generated}
The semigroup associated to the valuation $v_{\underline{w_0}}$ is finitely generated.
\end{corollary}
This is a consequence of Theorem \ref{thm crystal okounkov}, the observation above that the cone $\text{Cone}(\Delta_{\underline{w_0}}(\lambda))\subset \R \times \R^N$ over $\Delta_{\underline{w_0}}(\lambda)$ is a strongly convex rational polyhedral cone, and Gordon's Lemma.


\section{Proof of the main result} \label{main proof}
We aim to prove that the Gromov width of a generic coadjoint orbit $\mathcal{O}_{\lambda}$ of $\text{Sp}(n)$, passing through a point $\lambda$ in the interior of a chosen positive Weyl chamber and on a rational line, equipped with the Kostant-Kirillov-Souriau symplectic form,
is $$ \min \{ |\langle \lambda , \alpha^{\vee}  \rangle |;\ \alpha^{\vee} \textrm{ a coroot }\}.$$

Recall that all generic coadjoint orbits $\col$ are diffeomorphic to the flag manifold $G/B$, for $G=\text{Sp}(2n, \C)$.
For $i=1,\ldots, 2n$, let $\e_i: \mathfrak{sp}(2n,\C) \rightarrow \C$ denote the linear functional assigning to a matrix its $i$-th diagonal entry, $\e_i(x)=x_{ii}$. With this notation we can express the simple roots as:
\begin{equation}\label{list of roots}
\alpha_n=\e_1 - \e_2, \; \alpha_{n-1}=\e_2 - \e_3, \; \hdots, \; \alpha_2=\e_{n-1} - \e_n,\; \alpha_1=2\e_n.
\end{equation}
Note that the above enumeration is non-standard. We follow Littelmann's enumeration (\cite{L}), as we are going to quote some results from \cite{L}. All the roots are given by $\pm 2\e_i$ and $\pm(\e_i\pm \e_j), \; i \neq j$. The fundamental weights are $\omega_i=\e_1+\e_2+\hdots+\e_i$, $i=1,2,\hdots,n$ and each $\lambda \in \Lambda_\R^+$ can be expressed as:
$$\begin{array}{rl} \l&=\l_1\omega_1 + \l_2\omega_2 + \hdots + \l_n\omega_n \quad (\l_i \geq 0 ) \\
                    &=(\l_1+\l_2+\hdots + \l_n)\e_1 + (\l_2+\hdots+\l_n)\e_2+\hdots+\l_n\e_n.
\end{array}$$
Then 
$$ \min \{ |\langle \lambda , \alpha^{\vee}  \rangle |;\ \alpha^{\vee} \textrm{ a coroot }\}=\min \{ \lambda_1,\ldots, \lambda_n\}.$$

We first analyze the situation when
$\lambda $ {\it is integral}.
Then $\lambda$ is a dominant weight and thus there exists a very ample line bundle $\mathcal{L}_{\lambda}$ on $G/B$ whose space of sections $H^0(G/B, \mathcal{L}_{\lambda})$ is isomorphic to $V_{\lambda}^*$. 
 The very ample line bundle $\mathcal{L}_{\lambda}$ induces the Kodaira embedding $j_\lambda \colon G/B \hookrightarrow \P(H^0(G/B, \mathcal{L}_{\lambda})^*)$ and 
 one can use $j_\lambda$ to pull back the Fubini-Study symplectic structure from the projective space to $G/B$.
 The thus obtained symplectic manifold
 $(G/B, \omega_{\lambda}=j_\lambda^*(\omega_{FS})\,)$
 is symplectomorphic to $\col$ with the standard Kostant-Kirillov-Souriau symplectic structure.

As explained in Section \ref{tools} (page \pageref{our valuation}), a choice of a reduced decomposition $\underline{w_0}=(\alpha_{i_1},\ldots, \alpha_{i_N})$ of the longest word $w_0=s_{\alpha_{i_1}}\cdots s_{\alpha_{i_N}}$ in the Weyl group gives rise to a highest term valuation $v_{\underline{w_0}}$ with one-dimensional leaves,
  and to a semigroup $S$ with the associated Newton-Okounkov body $\Delta(S)$. This semigroup is finitely generated (Corollary \ref{finitely generated}).
  Theorems \ref{existence degeneration}, \ref{existence integrable system} and \ref{thm crystal okounkov} imply the following:
  
\begin{corollary}\label{cor toric action}
For integral $\lambda$, there exists a toric action on an open dense subset of $\col$. Its moment map image is the interior of the string polytope $\Delta_{\underline{w_0}}(\lambda)\subset \R^{n^2}$.
\end{corollary} 

We prove the main theorem by exhibiting an embedding of (a $GL(n^2,\Z)$ image of) a simplex $\Delta^{n^2}( \min \{ \lambda_1,\ldots, \lambda_n\})$, of size equal to $\min \{ \lambda_1,\ldots, \lambda_n\}$, in the string polytope $\Delta_{\underline{w_0}}(\lambda)$.
The polytope $\Delta_{\underline{w_0}}(\lambda)$ for the longest word decomposition
$$w_0=s_1(s_2s_1s_2)\ldots (s_{n-1}\ldots s_1 \ldots s_{n-1})(s_ns_{n-1}\ldots s_1 \ldots s_{n-1}s_n),$$
(where $s_j=s_{\alpha_j}$, with the numbering of the simple roots from \eqref{list of roots}), was described by Littelmann (\cite[Section 6, Theorem 6.1 and Corollary 6]{L}; note the misprint in Corollary 6: $\lambda_{m-j+1}$ should be $\lambda_j$ as can be deduced from \cite[Proposition 1.5]{L}).

\begin{proposition}\cite{L}\label{string polytope description}
Fix a dominant weight 
$$\lambda= \lambda_1 \omega_1+\ldots + \lambda_n \omega_n=(\lambda_1+\ldots+\lambda_n) \e_1+\ldots+ \lambda_n \e_n.$$
Then the associated string polytope $\Delta_{\underline{w_0}}(\lambda)$ is the convex polytope in $\R^{n^2}$ given by $n^2$-tuples $\{a_{i,j}\,|\,1\leq i \leq n,\,i \leq j \leq 2n-i\}$ which satisfy:
$$a_{i,i} \geq a_{i,i+1} \geq \ldots \geq a_{i,2n-i} \geq 0, \ \ \forall \ i=1,\ldots n,$$
\begin{align*}
\bar{a}_{i,j} &\leq \lambda_{j}+s(\bar{a}_{i,j-1}) -2s (a_{i-1,j}) +s(a_{i-1,j+1}),\\
a_{i,j} &\leq \lambda_{j}+s(\bar{a}_{i,j-1}) -2s (\bar{a}_{i,j}) +s(a_{i,j+1}),\\
a_{i,n} &\leq \lambda_n +s(\bar{a}_{i,n-1}) -s (a_{i-1,n}),
\end{align*}
for all $1\leq i,j\leq n$, where we use the following notation
$$\bar{a}_{i,j}:=a_{i,2n-j} \textrm{ for }1 \leq j\leq n,$$
and
$$s(\bar{a}_{i,j}):=\bar{a}_{i,j} + \sum_{k=1}^{i-1}(a_{k,j}+\bar{a}_{k,j}),\ s(a_{i,j}):= \sum_{k=1}^{i}(a_{k,j}+\bar{a}_{k,j}),$$
 for $j<n$ (so $s(a_{i,n})=2 \sum_{k=1}^{i}a_{k,n}$).
\end{proposition}
In the above formula we use the convention that $a_{i,j}=\bar{a}_{i,j}=0$ if $j<i$. Note that if $i>1$ then for $j<i$ the expression
$s(\bar{a}_{i,j})$ is not $0$ but equals $\sum_{k=1}^{i-1}(a_{k,j}+\bar{a}_{k,j})$.

Moreover, in \cite{L} Littelmann defines a map from $\R^{n^2}$ to $\R^{n^2}$ which maps $\Delta_{\underline{w_0}}(\lambda)$ to the polytope $GT(\lambda)$, obtained from a Gelfand-Tsetlin pattern\footnote{Remark on notation. Performing Thimm's trick for the sequence of subgroups $Sp(1) \subset \ldots \subset Sp(n-1) \subset \text{Sp}(n)$ produces a Hamiltonian action of a torus of dimension $\frac 1 2 n(n-1)$ on $\col$. The image of the momentum map for this torus (not toric) action is a polytope of dimension $\frac 1 2 n(n-1)$ which is sometimes called a Gelfand-Tsetlin polytope. This polytope can be obtained from $GT(\lambda)$ described here via a projection forgetting the $\{z_{i,j}\}$ coordinates.}, which induces a bijection between the integral points  of $\Delta_{\underline{w_0}}(\lambda)$ and $GT(\lambda)$. 
We first recall from \cite{L} the definition of the polytope $GT(\lambda)$. 
For simplicity of notation let 
$$l_j:= \lambda_j+\ldots+\lambda_n$$
so that $\lambda= l_1 \e_1 +\ldots+l_n \e_n$.
Let $\{y_{i,j}\}$, $2\leq i \leq n$, $i \leq j \leq n$ and $\{ z_{i,j}\}$, $1\leq i \leq n$, $i \leq j \leq n$, denote coordinates in $\R^{n^2}$. A point 
$$(y,z):=(z_{1,1},\ldots, z_{1,n}, y_{2,2},\ldots, y_{2,n},z_{2,2},\ldots, z_{2,n}, \ldots, y_{n,n}, z_{n,n})$$
in $\R_{\geq 0}^{n^2}$ is called a {\bf Gelfand-Tsetlin pattern} for $\lambda= l_1 \e_1 +\ldots+l_n \e_n$ if the entries satisfy the ``betweenness " condition:
\begin{equation}\label{eqn gtpolytope}
l_k \geq z_{1,k} \geq l_{k+1},\ \ \ 
z_{i-1,j-1}\geq y_{i,j}\geq z_{i-1,j} ,\ \ \ 
y_{i,j} \geq z_{i,j} \geq y_{i,j+1}
\end{equation}
for $1 \leq k \leq n$, $1\leq i \leq n$, $i \leq j \leq n$, where $y_{1,j}=l_j$ for simplicity of notation.
A convenient way to visualize these conditions is to organize the coordinates of $\R^{n^2}$ as in Figure \ref{gtcoordinates} (for $n=3$). The value of each coordinate must be between the values of its top right and top left neighbors.
\begin{center}
\begin{figure}[h]
\includegraphics[width=.2\textwidth]{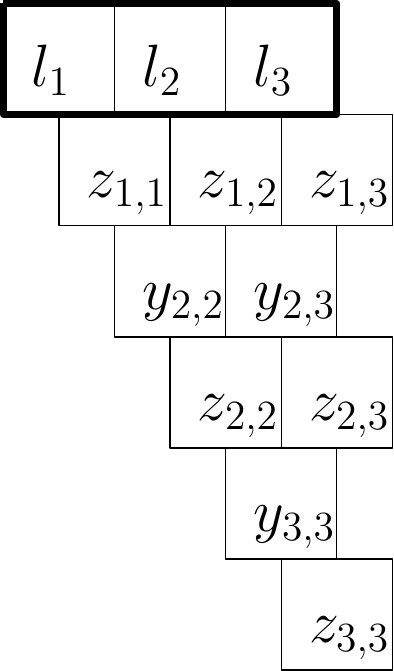}
\caption{ A graphical presentation of a Gelfand-Tsetlin pattern (for $n=3$).}
\label{gtcoordinates}
\end{figure}
\end{center}
Littelmann's map from the string polytope $\Delta_{\underline{w_0}}(\lambda)$ to the Gelfand-Tsetlin polytope $GT(\lambda)$  associates to each element $\underline{a} \in \R^{n^2}$ the pattern $P(\underline{a})=(y_{i,j},z_{i,j})$ of highest weight $\lambda=y_{1,1}\epsilon_1+ \hdots + y_{1,n}\epsilon_n$ defined by equations (\cite{L}; note the misprint in \cite{L}: $\alpha_{m-k+1}$ should be  $\alpha_{m-j+1}$):
\begin{align}
\label{eqn stringtoGT}
y_{i,1}\epsilon_1+ \hdots +y_{i,n}\epsilon_n &= \lambda-\sum^{i-1}_{k=1}{\left(a_{k,n}\alpha_1+\sum^{n-1}_{j=k}(a_{k,j}+\overline{a}_{k,j})\alpha_{n-j+1}\right)} 
\nonumber \\
\\
z_{i,1}\epsilon_1+\hdots + z_{i,n}\epsilon_n &= \sum^{n}_{k=1}y_{i,k}\epsilon_k - \frac{a_{i,n}}{2}\alpha_1 - \sum^{n-1}_{j=i}\overline{a}_{i,j}\alpha_{n-j+1}, \nonumber
\end{align}
where $\alpha_j$ are the simple roots as in \eqref{list of roots}:
$$\alpha_n=\e_1 - \e_2, \; \alpha_{n-1}=\e_2 - \e_3, \; \hdots, \; \alpha_2=\e_{n-1} - \e_n,\; \alpha_1=2\e_n.$$
In fact this map is a $GL(n^2, \Z)$ transformation followed by a translation, as we now show.

\begin{proposition}\label{c is gt}
The map \eqref{eqn stringtoGT} which maps the polytope $\Delta_{\underline{w_0}}(\lambda)$  to the Gelfand-Tsetlin polytope $GT(\lambda)$ is a $GL(n^2, \Z)$-transformation followed by a translation.
\end{proposition}

We are grateful to the referee for suggesting to replace our original proof (by direct computation) with the following one.
\begin{proof}
Clearly equation \eqref{eqn stringtoGT} defines a composition of a linear map $\Phi \in GL(n^2,\R)$, defined by a matrix with integral entries (remember that $\alpha_1=2\epsilon_n$) and a translation. It suffices to show that $|\det \Phi|=1$ as this will imply that $\Phi^{-1}$ is also a matrix with integral entries, proving that $\Phi \in GL(n^2,\Z)$.
The fact that \eqref{eqn stringtoGT} is a bijection between integral points of $\Delta_{\underline{w_0}}(k\lambda)=k\Delta_{\underline{w_0}}(\lambda)$ and integral points of $GT(k\lambda)=kGT(\lambda)$ for any $k \in \N$, together with the fact for any integral polytope $\Delta \in \R^{n^2}$, its volume is the limit
$$vol\, (\Delta) =\lim_{k\rightarrow \infty } \frac{\# (k\Delta \cap \Z^{n^2})}{k^{n^2}},$$
implies that $vol\,(\Delta_{\underline{w_0}}(\lambda))=vol\,GT(\lambda)$. Therefore, we must have that $|\det \Phi|=1$.
\end{proof}

\begin{example} Let's take a closer look at the case $n=2$ and re-prove the above Proposition by direct computation.  
In this case, the simple roots are: $\alpha_1=2\epsilon_2$, $\alpha_2=\epsilon_1-\epsilon_2$. 
We fix a reduced word decomposition $w_0=s_1 \,s_2\,s_1\,s_2$,
and fix a weight: 
$$\lambda=\lambda_1 w_1 + \lambda_2 w_2=(\lambda_1 +\lambda_2) \epsilon_1 + \lambda_2 \epsilon _2.$$
The associated string polytope $\Delta=\Delta_{\underline{w_0}}(\lambda)$ is a subset of $\R^4$, for which we use coordinates $a_{22}, a_{11}, a_{12}, a_{13}$,
and is defined by the following inequalities:
$$a_{22}\geq 0,\ a_{11}\geq a_{12}\geq a_{13}\geq 0,$$
\begin{align*}
a_{13}=\bar{a}_{11} &\leq \lambda_{1},\\
a_{11} &\leq \lambda_{1} -2s (\bar{a}_{11}) +s(a_{12})=\lambda_{1} -2a_{13}+ 2a_{12},\\
a_{12} &\leq \lambda_2 +s(\bar{a}_{11})=\lambda_2+a_{13},\\
a_{22}& \leq \lambda_2+s(\bar{a}_{21})-s(a_{12})=\lambda_2+a_{11} +a_{13}-2a_{12}.
\end{align*}
We derive the second set of inequalities for the symplectic group (see also Corollary 6 of \cite{L}) from the description of the string polytope for a general $G$ given in \cite{L}: the Definition on page $5$ and Proposition $1.5$. According to this description (using our fixed reduced word decomposition and numbering of simple roots):
\begin{align*}
a_{13} &\leq \langle \lambda, \alpha_2^{\vee} \rangle =\langle \lambda, (\epsilon_1-\epsilon_2)^{\vee} \rangle=
 (\lambda_{1}+\lambda_2) - \lambda_2=\lambda_1,\\
 a_{12} &\leq \langle \lambda - a_{13}\alpha_2, \alpha_1^{\vee} \rangle =  
 \langle \lambda, 2\epsilon_2^{\vee} \rangle-a_{13} \langle\epsilon_1-\epsilon_2, 2\epsilon_2^{\vee} \rangle
 =\lambda_2+a_{13},\\
a_{11} &\leq \langle \lambda - a_{13}\alpha_2-a_{12} \alpha_1, \alpha_2^{\vee} \rangle  \\ & 
= \langle \lambda,  (\epsilon_1-\epsilon_2)^{\vee} \rangle-a_{13} \langle\epsilon_1-\epsilon_2,  (\epsilon_1-\epsilon_2)^{\vee} \rangle
 -a_{12} \langle  2\epsilon_2, (\epsilon_1-\epsilon_2)^{\vee} \rangle
 \\ &=\lambda_1-2a_{13} -a_{12} (-2),\\
a_{22}&\leq \langle \lambda - a_{13}\alpha_2-a_{12} \alpha_1-a_{11} \alpha_2, \alpha_1^{\vee} \rangle  \\ & 
= \lambda_2+a_{13} - a_{12} \langle 2\epsilon_2,2\epsilon_2^{\vee} \rangle-a_{11} \langle \epsilon_1-\epsilon_2,2\epsilon_2^{\vee} \rangle
\\&=\lambda_2+a_{13} -2 a_{12}+a_{11}.
\end{align*}
We now analyze the map from the above string polytope to the Gelfand-Tsetlin polytope, given by equations \eqref{eqn stringtoGT}.
As
$$z_{11}\epsilon_1+z_{12} \epsilon_2=(\lambda_1+\lambda_2) \epsilon_1 +\lambda_2 \epsilon_2 - \frac{a_{12}}{2} (2 \epsilon_2)-a_{13}(\epsilon_1-\epsilon_2),$$
we get that:
\begin{align*}
z_{11}&=\lambda_1 + \lambda_2 - a_{13},\\
z_{12}&= \lambda_2 - a_{12}+ a_{13}.
\end{align*}
The value of $y_{22}$ is the coefficient of $\epsilon_2$ in
$\lambda- a_{12}(2 \epsilon_2) -(a_{11}+a_{13})(\epsilon_1-\epsilon_2)$, 
and $z_{22}$ is the coefficient of $\epsilon_2$ in 
$ y_{21}\epsilon_1+ y_{22}\epsilon_2 - \frac{a_{22}}{2}(2 \epsilon_2)$
thus
\begin{align*}
y_{22}&=\lambda_2+a_{11}-2a_{12}+a_{13},\\
z_{22}&=y_{22}-a_{22},
\end{align*}
i.e.
\begin{displaymath}
\left[\begin{array}{c}
z_{11}\\
z_{12}\\
y_{22}\\
z_{22}
\end{array}
\right]
=
\left[
\begin{array}{rrrr}
0 &0& 0&-1\\
0& 0&-1&1\\
0&1&-2&1\\
-1&1&-2&1
\end{array} \right]
\cdot
\left[\begin{array}{c}
a_{22}\\
a_{11}\\
a_{12}\\
a_{13}
\end{array}
\right]
+
\left[\begin{array}{c}
\lambda_1+\lambda_2\\
\lambda_2\\
\lambda_2\\
\lambda_2
\end{array}
\right]
\end{displaymath}

Therefore, the inequalities describing the string polytope translate to the following inequalities:
$$a_{22}\geq 0 \Leftrightarrow y_{22}\geq z_{22},$$
$$a_{11}\geq a_{12} \Leftrightarrow y_{22}+2a_{12}-a_{13}-\lambda_2 \geq a_{12}
 \Leftrightarrow y_{22}\geq -a_{12}+a_{13}+\lambda_2=z_{12},$$
$$a_{12}\geq a_{13} \Leftrightarrow 0\leq \lambda_2 - z_{12},$$
$$a_{13}\geq 0 \Leftrightarrow \lambda_1+\lambda_2\geq z_{11},$$
$$a_{13} \leq \lambda_1 \Leftrightarrow z_{11} \geq \lambda_2,$$
$$a_{12}-a_{13} \leq \lambda_2 \Leftrightarrow \lambda_2  -z_{12} \leq \lambda_2\Leftrightarrow 0\leq z_{12},$$
$$a_{11}-2a_{12}+2a_{13} \leq \lambda_1 \Leftrightarrow y_{22}-z_{11}+\lambda_1\leq \lambda_1 \Leftrightarrow y_{22}\leq z_{11},$$
$$a_{22}-a_{11}+2a_{12}-a_{13} \leq \lambda_2 \Leftrightarrow \lambda_2-z_{22}\leq \lambda_2 \Leftrightarrow 0\leq z_{22}.$$
The inequalities on the right are exactly the inequalities describing the Gelfand-Tsetlin polytope.
\end{example}

\begin{theorem}\label{gt fits simplex}
Let $r=\min \{ \lambda_1,\ldots, \lambda_n\}$ and $\Delta(r)$ be an $n^2$-dimensional simplex of size (the lattice length of the edges) $r$.
There exist $\Psi \in GL(n^2, \Z)$ and $x \in \R^{n^2}$ such that 
$$\Psi(\Delta(r)) +x \subset GT(\lambda).$$
\end{theorem}

\begin{proof}
Recall from \eqref{eqn gtpolytope} the definition of $GT(\lambda)$. Let $V_0:=V_0(\lambda)$ be a vertex of $GT(\lambda)$ where all the coordinates $y_{i,j}$, $z_{i,j}$ are equal to their upper bounds, i.e.
$$z_{i,j}=y_{i,j}=z_{i-1,j-1}=y_{i-1,j-1}=\ldots=z_{1,j-i+1}=l_{j-i+1}.$$
We will analyze the edges starting from $V_0$. To obtain an edge starting from $V_0$, we pick one of the inequalities  \eqref{eqn gtpolytope} defining $GT(\lambda)$ which is an equality at $V_0$, and consider the set of points in $GT(\lambda)$ satisfying all the same equations that $V_0$ satisfies, except possibly this chosen one. More precisely, each of the $\frac 1 2 n(n-1)$ pairs $(i_0,j_0)$ with $2\leq i_0 \leq j_0 \leq n$ gives us an edge $E_{i_0,j_0}$ defined as the set of points $(y,z) \in \R^{n^2}$ satisfying:
 \begin{align*}
& y_{i,j}=z_{i,j}=l_{j-i+1} \textrm{ unless } j-i=j_0-i_0  \textrm{ and } i\geq i_0,\\
  &y_{i_0,j_0}=z_{i_0,j_0}=y_{i_0+1,j_0+1}=\ldots=z_{n-j_0+i_0,n} \in [l_{j_0-i_0+2},l_{j_0-i_0+1}].
  \end{align*}
  The lattice length of this edge is $l_{j_0-i_0+1}-l_{j_0-i_0+2}=\lambda_{j_0-i_0+1}$. An example of such an edge is presented in Figure \ref{edges}, on the left.
  
  Moreover, each of the $\frac 1 2 n(n+1)$ pairs $(i_0,j_0)$ with $1\leq i_0 \leq j_0 \leq n$ gives us an edge $F_{i_0,j_0}$ defined as the set of points $(y,z) \in \R^{n^2}$ satisfying
 \begin{align*}
& y_{i,j}=z_{i,j}=l_{j-i+1} \textrm{ unless } j-i=j_0-i_0  \textrm{ and } i\geq i_0,\\
&y_{i_0,j_0}=l_{j_0-i_0+1},\\
 &z_{i_0,j_0}=y_{i_0+1,j_0+1}=z_{i_0+1,j_0+1}=\ldots=z_{n-j_0+i_0,n} \in [l_{j_0-i_0+2},l_{j_0-i_0+1}].
  \end{align*}
  
  The lattice length of this edge is also $l_{j_0-i_0+1}-l_{j_0-i_0+2}=\lambda_{j_0-i_0+1}$. An example of such an edge is presented in Figure \ref{edges}, on the right.
  \begin{figure}[h]
 \includegraphics[width=.7\textwidth]{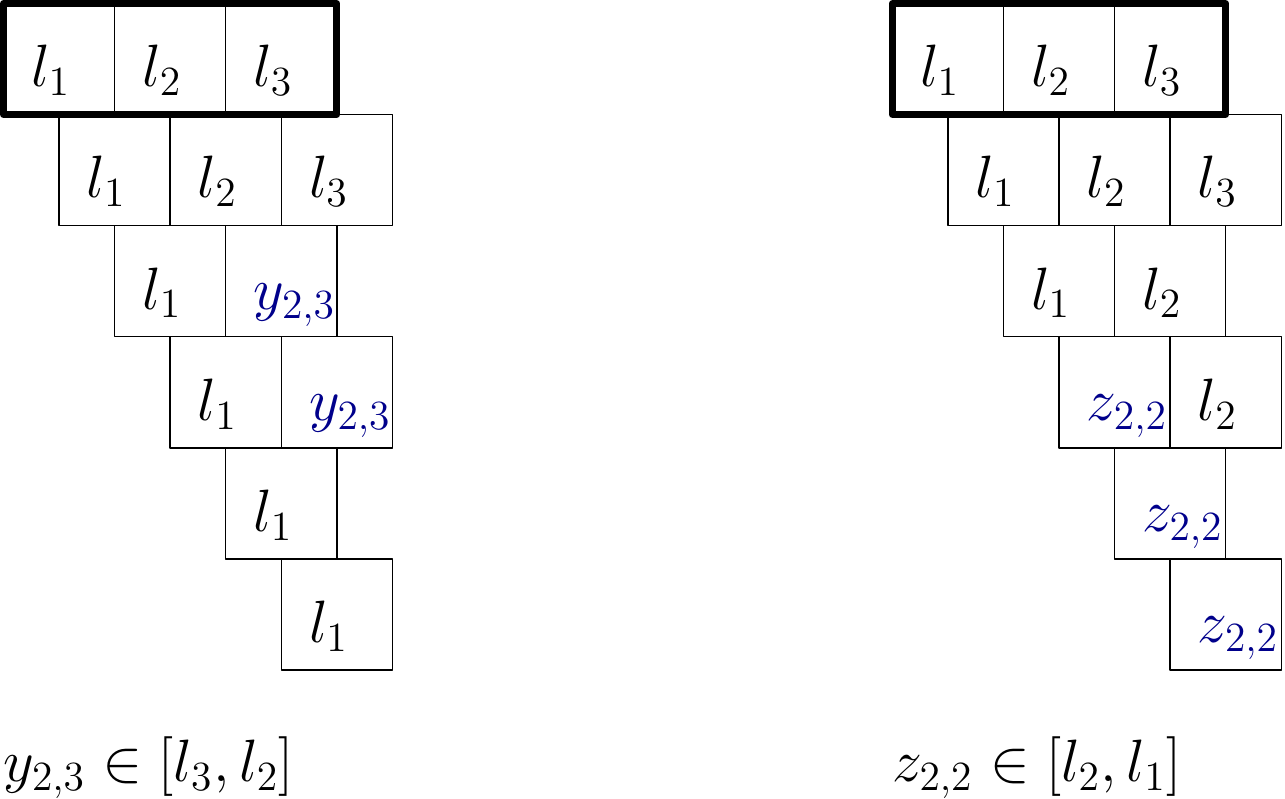}
\caption{The edges $E_{2,3}$ and $F_{2,2}$.}
\label{edges}
\end{figure}
  
  The above collection gives $\frac 1 2 n(n-1) + \frac 1 2 n(n+1)=n^2$ edges.
  Observe that the directions of these $n^2$ edges from $V_0$ form a $\Z$-basis of $\Z^{n^2} \subset \R^{n^2}$. 
  Indeed, if we keep the ordering: 
  $$z_{1,1},\, z_{1,2}, \ldots,z_{1,n},\, y_{2,2},\, y_{2,3}, \ldots,y_{2,n},\,z_{2,2}, \ldots,z_{2,n},...$$
  of our usual coordinates on $\R^{n^2}$ and
   order the edge generators by: 
   $$F_{1,1}, F_{1,2}, \ldots,F_{1,n}, E_{2,2}, E_{2,3}, \ldots,E_{2,n},F_{2,2}, \ldots,F_{2,n},...$$
   then the matrix of edge generators expressed in our usual basis is an upper triangular matrix with $(-1)$'s on the diagonal.
  Therefore, there
exist $\Psi \in GL(n^2, \Z)$ and $x \in \R^{n^2}$ such that: 
$$\Psi(\Delta(\min\{\lambda_j\ |\ j=1,\ldots,n\})) +x \subset GT(\lambda).$$
\end{proof}

Combining the above claims, we prove our main result.
\begin{proof}[Proof of Theorem \ref{gromovmain}.]
Let $$\lambda= \lambda_1 \omega_1+\ldots + \lambda_n \omega_n=(\lambda_1+\ldots+\lambda_n) \e_1+\ldots+ \lambda_n \e_n$$
 be a point in the interior of the chosen Weyl chamber $\Lambda^+_{\R}$ for the symplectic group $\text{Sp}(n)$, which lies on some rational line.
We want to show that the Gromov width of the coadjoint orbit $\col$ through $\lambda$ is at least  
$\min \{ \lambda_1,\ldots, \lambda_n\}.$ 

Recall that $\Lambda^+$ denotes the integral points of the positive Weyl chamber and let $\Lambda^+_{\Q}$ denote the rational ones.
If $\lambda $ is integral then, by Corollary \ref{cor toric action}, an open dense subset of $\col$ is equipped with a toric action. 
The momentum map image is the interior of a polytope equivalent under the action of $GL(n^2,\Z)$ and a translation to the 
Gelfand-Tsetlin polytope $GT(\lambda)$ (see Theorem \ref{string polytope description} and Proposition \ref{c is gt}).
Then Theorem \ref{gt fits simplex} and Proposition \ref{embedding} together with Theorem \ref{caviedes} prove that the Gromov width of $\col$ is exactly $\min \{ \lambda_1,\ldots, \lambda_n\}.$

If $\lambda$ is not integral, let $a \in \R_+$ be such that $a \lambda $ is integral. 
Observe that the coadjoint orbits $\mathcal{O}_{a \lambda}$ and $\col$ are diffeomorphic and differ only by a rescaling of their symplectic forms. 
Thus the Gromov width of $\mathcal{O}_{a \lambda}$, which is $\min \{ a\lambda_1,\ldots, a\lambda_n\}$, is $a$ times bigger than the Gromov width of $\col$. This proves that the Gromov with of $\col$ for $\lambda$ rational is exactly $\min \{ \lambda_1,\ldots, \lambda_n\}$.
\end{proof}

\subsection{Further comments} Note that the Gromov width of $\mathcal{O}_{\lambda}$ is lower semi-continuous as a function of $\lambda$, which one can prove by adjusting a ``Moser type" argument from  \cite{MP}. However, to extend our result to orbits $\mathcal{O}_{\lambda}$ with arbitrary $\lambda$, what is in fact needed is upper semi-continuity. We are very grateful to the referee for this remark. 
It is not known in general if the Gromov width of $\mathcal{O}_{\lambda}$ is upper semi-continuous. It would be if, for example, all obstructions to embeddings of balls come from $J$-holomorphic curves. (The last condition is often called the ``Biran Conjecture".) Note that an implication of the above conjecture of Biran is that the Gromov width of integral symplectic manifolds must be greater than or equal to one. This statement was proved, under certain assumptions, by Kaveh in \cite{K}.

\end{document}